\documentclass[12pt,twoside,reqno]{amsart}
\usepackage[latin1]{inputenc}
\usepackage[italian,english]{babel}
\usepackage{amssymb,latexsym}
\usepackage{amsmath,amsfonts,amsthm}
\usepackage{paralist}
\usepackage{color}

\numberwithin{equation}{section}

\newcommand{\R}{\mathbb{R}}
\newcommand{\N}{\mathbb{N}}

\newcommand{\I}{\mathcal{I}}
\newcommand{\J}{\mathcal{J}}
\newcommand{\T}{\mathcal{T}}
\newcommand{\V}{\mathcal{V}}

\DeclareMathOperator{\dive}{div}

\newtheorem{lem}{Lemma}
\newtheorem{thm}{Theorem}

\theoremstyle{remark}
\newtheorem{remark}{Remark}

\begin{document}

\title{Ground states for a fractional scalar field problem with critical growth}
\author{Vincenzo Ambrosio}
\address{Dipartimento di Matematica e Applicazioni "R. Caccioppoli"\\
         Universit\`a degli Studi di Napoli Federico II\\
         via Cinthia, 80126 Napoli, Italy}
\email{vincenzo.ambrosio2@unina.it}

\maketitle
\begin{abstract}
We prove the existence of a ground state solution for the following fractional  scalar field equation
\begin{align*}
(-\Delta)^{s} u= g(u) \mbox{ in } \R^{N} 
\end{align*}
where $s\in (0,1)$, $N> 2s$, $(-\Delta)^{s}$ is the fractional Laplacian, and $g\in C^{1, \beta}(\R, \R)$ is an odd function satisfying the critical growth assumption.
\end{abstract}

\section{Introduction}

\noindent
The aim of this paper is  to study the following fractional scalar field equation 
\begin{align}\label{P}
(-\Delta)^{s} u= g(u) \mbox{ in } \R^{N} 
\end{align}
with $s\in (0,1)$, $N> 2s$, and $g: \R\rightarrow \R$ is a smooth function verifying some suitable growth conditions. Here $(-\Delta)^{s}$ is the fractional Laplacian which can be defined for any function $u$ in the Schwarz class as
$$
(-\Delta)^{s}u(x)=C_{N,s} \lim_{\epsilon \rightarrow 0} \int_{\R^{N}\setminus B_{\epsilon}(x)} \frac{u(x)-u(y)}{|x-y|^{N+2s}} dy  \quad (x\in \R^{N})
$$
where $C_{N,s}$ is a dimensional constant depending only on $N$ and $s$, whose value can be found in \cite{DPV}.\\
A basic motivation for the study of (\ref{P}) comes from looking for standing waves $\psi(x, t)=u(x)e^{-\imath ct}$ for the fractional Schr\"odinger equation 
$$
\imath \frac{\partial \psi}{\partial t}=(-\Delta)^{s} \psi+g(\psi) \quad (t, x)\in \R\times \R^{N}.
$$
Such equation has been proposed by Laskin \cite{Laskin1, Laskin2}, and comes from an expansion of the Feynman path integral from Brownian-like to Levy--like quantum mechanical paths. In Laskin's studies, the Feynman path integral leads to the classical Schr\"odinger equation and the path integral over L\`evy trajectories leads to the fractional Schr\"odinger equation.\\
When $s=1$, (\ref{P}) corresponds to the classical Schr\"odinger equation $-\Delta u=g(u)$ in $\R^{N}$, which has been extensively studied by several authors, and we cannot review the huge literature here.\\
Nowadays, there are many papers dealing with the existence of solutions for the fractional Schr\"odinger equation.
Di Pierro et al. \cite{DPPV} proved the existence of a positive and spherically symmetric solution to (\ref{P}) when $g(u)=-u+|u|^{p-1}u$ with $1<p<\frac{N+2s}{N-2s}$. 
Felmer et al. \cite{FQT} investigated existence, regularity, decay and symmetry properties of positive solution to $(-\Delta)^{s}u+u=f(x, u)$, when $f$ has subcritical growth and satisfies the Ambrosetti-Rabinowitz condition. Secchi \cite{Secchi1} showed the existence of ground state solutions for a nonlinear Schr\"odinger equation with an external potential, via minimization on the Nehari manifold. Coti Zelati and Nolasco \cite{CN} obtained the existence of a ground state of some fractional Schr\"odinger equation involving the operator $(-\Delta+m^{2})$ with $m>0$. D\'avila et al. \cite{DDW} considered the existence and concentration phenomena of multi-peak solutions for a fractional Schr\"odinger equation, by using Lyapunov-Schmidt reduction method.
In \cite{A} and \cite{CW}, it has been established the existence and the multiplicity of ground state solutions when the nonlinearity $g$ in (\ref{P}) satisfies the following Berestycki-Lions \cite{BL1} type assumptions
\begin{compactenum}[({\it H}1)]
\item $\displaystyle{-\infty <\liminf_{t\rightarrow 0^{+}} \frac{g(t)}{t} \leq \limsup_{t\rightarrow 0^{+}} \frac{g(t)}{t}=-a<0}$;
\item $\displaystyle{-\infty <\limsup_{t\rightarrow +\infty} \frac{g(t)}{t^{2^{*}_{s}-1}} \leq 0}$, where ${\displaystyle{2^{*}_{s}=\frac{2N}{N-2s}}}$;
\item There exists  $\xi_{0}>0$ such that  $\displaystyle{G(\xi_{0})= \int_{0}^{\xi_{0}} g(\tau) d\tau >0}$.
\end{compactenum}
\smallskip

\noindent
Motivated by the last two papers, in the present paper we aim to investigate the existence of least energy solutions for the equation (\ref{P}), when $g$ satisfies a critical growth.
In this case, the main difficulty is due to the lack of compactness of the embedding of $H^{s}(\R^{N})$ into $L^{2^{*}_{s}}(\R^{N})$.\\
The specific assumptions we are considering on $g: \R \rightarrow \R$ are now listed as follows:
\begin{compactenum}[({\it g}1)]
\item $g\in C^{1, \beta}(\R, \R)$ for some $\beta>\max\{0, 1-2s\}$ and $g$ is odd;
\item $\displaystyle{\lim_{t\rightarrow 0^{+}} \frac{g(t)}{t}=-a<0}$;
\item $\displaystyle{\lim_{t\rightarrow +\infty} \frac{g(t)}{t^{2^{*}_{s}-1}}= b>0}$, where $\displaystyle{2^{*}_{s}=\frac{2N}{N-2s}}$;
\item There exist  $C>0$ and $\displaystyle{\max\left\{2, \frac{4s}{N-2s}\right\}<q<2^{*}_{s}}$ such that  
$$
g(t)-bt^{{2^{*}_{s}}-1} +a t \geq C t^{q-1} \mbox{ for all } t>0.
$$
\end{compactenum}
We point out that the regularity of $g$ is higher than in \cite{BL1}, and this seems to be due to the more demanding assumptions for ``elliptic" regularity in the framework of fractional operators.
In the classic setting, that is $s=1$, the assumptions $(g3)$ and $(g4)$ have been used in \cite{ZZ} to study ground state solutions for $-\Delta u=g(u)$ in $\R^{N}$, when $g$ has critical growth.
\begin{remark}
Let us observe that if $g(t)=b|t|^{2^{*}_{s}-2}t-at$, then $g$ satisfies $(g1)$-$(g3)$ but does not verify $(g4)$. Moreover, in view of the Pohozaev identity \cite{CW}, it is not difficult to prove that (\ref{P}) does not admit solution. Then, without $(g4)$, the assumptions $(g1)$-$(g3)$ are not sufficient to guarantee the existence of a ground state to (\ref{P}). 
\end{remark}

\noindent
Our first main result can be stated as follows:
\begin{thm}\label{thm1}
Let $s\in (0,1)$ and $N> 2s$. Assume that  $g$ verifies $(g1)$-$(g4)$. Then (\ref{P}) possesses a radial positive least energy solution $\omega\in H^{s}(\R^{N})$. 
\end{thm}

\noindent
We recall that for a weak solution of problem (\ref{P}) we mean a function $u\in H^{s}(\R^{N})$ such that 
$$
\iint_{\R^{2N}} \frac{(u(x)-u(y))}{|x-y|^{N+2s}} (\varphi(x)-\varphi(y)) \, dx dy=\int_{\R^{N}} g(u) \varphi \, dx
$$
for any $\varphi \in H^{s}(\R^{N})$.\\
One of the main difficulty in studying (\ref{P}) is the nonlocal character of the fractional Laplacian $(-\Delta)^{s}$ with $s\in (0,1)$. To circumvent this problem, Caffarelli and Silvestre \cite{CS} showed that it is possible to realize $(-\Delta)^{s}$  as an operator that maps a Dirichlet boundary condition to a Neumann boundary condition via an extension degenerate elliptic problem in the upper-half space $\R^{N+1}_{+}$. 
Anyway, in this work we prefer to analyze (\ref{P}) directly in $H^{s}(\R^{N})$ in order to adapt the variational techniques used in the classic framework.
More precisely, we will prove our results following the approaches in  \cite{JT} and \cite{ZZ}.\\
Let us introduce the following functional on $H^{s}(\R^{N})$
$$
\I(u)=\frac{1}{2} \iint_{\R^{2N}} \frac{|u(x)-u(y)|^{2}}{|x-y|^{N+2s}} dx dy-\int_{\R^{N}} G(u) dx=:\T(u)-\V(u),
$$
and we look for critical points of $\I$.
In order to do this, we consider the constraint minimization problem
 \begin{equation*}
M:=\inf \left \{\T(u): \V(u)=1, u\in H^{s}(\R^{N})\right\}.
\end{equation*}
By assumptions on $g$, it follows that $M$ satisfies the following bounds:
$$
0<M<\frac{1}{2}(2^{*}_{s})^{\frac{N-2s}{N}}S_{*}
$$
where $S_{*}$ is the best constant in the fractional Sobolev embedding $H^{s}(\R^{N})$ in $L^{2^{*}_{s}}(\R^{N})$.
Thanks to this information, we can see that the above minimization problem admits a positive and radially symmetric minimizer in $H^{s}(\R^{N})$.
At this point, we use some arguments similar to those developed in \cite{JT}, to prove that $\I$ admits a radial positive critical point $\omega\in H^{s}(\R^{N})$, which is a least energy solution to (\ref{P}).\\
\noindent
Finally, observing that $\I$ has a mountain pass geometry, we are also able to prove the following result in the spirit of \cite{JT}:
\begin{thm}\label{thm2}
Let $\omega$ be the least energy solution in Theorem \ref{thm1}. Then, there exists a path $\gamma \in \Gamma$, with $\Gamma=\{\gamma\in C([0, 1], H^{s}(\R^{N})): \gamma(0)=0 \mbox{ and } \I(\gamma(1))<0\}$, such that $\omega \in \gamma([0,1])$ and $\max_{t\in [0,1]} \I(\gamma(t))= \I(\omega)$. Moreover, the Mountain Pass value
\begin{equation}\label{MPvalue}
c=\inf_{\gamma\in \Gamma} \max_{t\in [0, 1]} \I(\gamma(t))
\end{equation}
coincides with 
\begin{align}\label{LEL}
m=\inf \left\{\I(u) : u\in H^{s}(\R^{N})\setminus \{0\} \mbox{ is a solution of (\ref{P})}\right\}.
\end{align}
\end{thm}

\noindent
Recently, a great interest has been focused on the study of problems involving fractional powers of the Laplacian.
This type of problems arises in many different applications, such as, the optimization, finance, phase transitions,  anomalous diffusion, crystal dislocation, conservation laws, ultra-relativistic limits of quantum mechanics, quasi-geostrophic flows, minimal surfaces and water waves. 
The literature is too wide to attempt a reasonable list of references here, so we derive the interested reader to the work by Di Nezza et al. \cite{DPV}, where a more extensive bibliography and an introduction to the subject are given.\\
\noindent
The organization of this paper is as follows: in Section $2$ we give some useful results which we will use frequently along the paper; in Section $3$ we prove that (\ref{P}) admits a least energy solution; finally in Section $4$ we show that the mountain pass value defined in (\ref{MPvalue}), gives the least energy level.

\section{Preliminaries}

\noindent
In this section we collect a few results that we are later going to use for the proof of the main results. \\
For any $s\in (0,1)$ we define the fractional Sobolev space
$$
H^{s}(\R^{N})= \left\{u\in L^{2}(\R^{N}) : \frac{|u(x)-u(y)|}{|x-y|^{\frac{N+2s}{2}}} \in L^{2}(\R^{2N}) \right \}
$$
endowed with the natural norm 
$$
\|u\|_{H^{s}(\R^{N})} = \sqrt{[u]_{H^{s}(\R^{N})}^{2} + \|u\|_{L^{2}(\R^{N})}^{2}}
$$
where the term
$$
[u]^{2}_{H^{s}(\R^{N})} =\iint_{\R^{2N}} \frac{|u(x)-u(y)|^{2}}{|x-y|^{N+2s}} \, dx \, dy
$$
is the so-called Gagliardo seminorm of $u$. \\

\noindent
For convenience of the reader we recall from \cite{DPV} the following:
\begin{thm}\label{Sembedding}
Let $s\in (0,1)$ and $N>2s$. Then there exists a sharp constant $S_{*}=S(N, s)>0$, whose exact value can be found in \cite{CT}, 
such that for any $u\in H^{s}(\R^{N})$
\begin{equation}\label{FSI}
\|u\|^{2}_{L^{2^{*}_{s}}(\R^{N})} \leq S_{*} [u]^{2}_{H^{s}(\R^{N})}. 
\end{equation}
Moreover $H^{s}(\R^{N})$ is continuously embedded in $L^{q}(\R^{N})$ for any $q\in [2, 2^{*}_{s}]$ and compactly in $L^{q}_{loc}(\R^{N})$ for any $q\in [2, 2^{*}_{s})$. 
\end{thm}

\begin{remark}
The exact value of the best constant $S_{*}$ appearing in (\ref{FSI}), has been calculated explicitly in \cite{CT}. Moreover, the authors proved that the equality in (\ref{FSI}) holds if and only if 
$$
u(x)=c(\mu^{2}+(x-x_{0})^{2})^{-\frac{N-2s}{2}}
$$
where $c\in \R$, $\mu>0$ and $x_{0}\in \R^{N}$ are fixed constants.
\end{remark}
\noindent
Now we introduce the space of radial functions in $H^{s}(\R^{N})$
$$
H^{s}_{r}(\R^{N})=\left \{u\in H^{s}(\R^{N}): u(x)=u(|x|)\right\}. 
$$
Related to this space, the following compactness result due to Lions \cite{Lions} holds:
\begin{thm}\cite{Lions}\label{Lions}
Let $s\in (0,1)$ and $N\geq 2$. Then $H^{s}_{r}(\R^{N})$ is compactly in $L^{q}(\R^{N})$ for any $q\in (2, 2^{*}_{s})$.
\end{thm}

\noindent
Finally, it is worth recalling the following Lemma proved in \cite{CW} which, in some sense, substitutes the Strauss's compactness lemma:
\begin{lem}\cite{CW}\label{CW}
Let $(X, \|\cdot\|)$ be a Banach space such that $X$ is embedded respectively continuously and compactly into $L^{q}(\R^{N})$ for $q\in [q_{1}, q_{2}]$ and $q\in (q_{1}, q_{2})$, where $q_{1}, q_{2}\in (0, \infty)$.
Assume that $(u_{k})\subset X$, $u: \R^{N} \rightarrow \R$ is a measurable function and $P\in C(\R, \R)$ is such that
\begin{compactenum}[(i)]
\item $\displaystyle{\lim_{|t|\rightarrow 0} \frac{P(t)}{|t|^{q_{1}}}=0}$, \\
\item $\displaystyle{\lim_{|t|\rightarrow \infty} \frac{P(t)}{|t|^{q_{2}}}=0}$,\\
\item $\displaystyle{\sup_{k\in \N} \|u_{k}\|_{X}<\infty}$,\\
\item $\displaystyle{\lim_{k \rightarrow \infty} P(u_{k}(x))=u(x)} \mbox{ for a.e. } x\in \R^{N}$.
\end{compactenum}
Then, up to a subsequence, we have
$$
\lim_{k\rightarrow \infty} \|P(u_{k})-u\|_{L^{1}(\R^{N})}=0.
$$
\end{lem}

\section{Proof of Theorem \ref{thm1}}
\noindent
This section is dedicated to the proof of Theorem \ref{thm1}. Without loss of generality, we will assume that $b=1$ in $(g3)$.\\
In order to study weak solutions to (\ref{P}), we look for critical points of the following functional
$$
\I(u)= \T(u)-\V(u)
$$
for $u\in H^{s}(\R^{N})$, where 
$$
\mathcal{T}(u)=\frac{1}{2} [u]^{2}_{H^{s}(\R^{N})} \mbox{ and } \V(u)=\int_{\R^{N}} G(u) dx. 
$$
By Theorem \ref{Sembedding} and assumptions on $g$, it is clear that $\I$ is well defined, and that $\I\in C^{1}(H^{s}(\R^{N}))$.  \\
We begin proving the following 
\begin{lem}\label{lem1}
Let $M=\inf \left\{\T(u): \V(u)=1, u\in H^{s}(\R^{N})\right\}$. \\
Then $0<M<\frac{1}{2}(2^{*}_{s})^{\frac{N-2s}{N}}S_{*}$.
\end{lem}
\begin{proof}
Firstly we prove that $\{u\in H^{s}(\R^{N}) : \V(u)=1\}$ is not empty.
Fix $\eta \in C^{\infty}_{0}(\R^{N})$ a cut-off function with support in $B_{2}$ and such that $0\leq \eta \leq 1$, and $\eta=1$ on $B_{1}$, where $B_{r}$ denotes the ball in $\R^{N}$ of center at origin and radius $r$.
For $\epsilon>0$, let us define $\psi_{\epsilon}(x)=\eta(x)U_{\epsilon}(x)$, where
$$
U_{\epsilon}(x)=\frac{\kappa \epsilon^{\frac{N-2s}{2}}}{(\epsilon^{2}+|x|^{2})^{\frac{N-2s}{2}}}
$$
is a solution to
$$
(-\Delta)^{s}u=S_{*}|u|^{2^{*}_{s}-2}u \mbox{ in } \R^{N}
$$
and $\kappa$ is a suitable positive constant depending only on $N$ and $s$.\\
Now we set
$$
v_{\epsilon}(x)=\frac{\psi_{\epsilon}}{\|\psi_{\epsilon}\|_{L^{2^{*}_{s}}(\R^{N})}}.
$$
By performing similar calculations to those in \cite{SV} (see Proposition $21$ and $22$), we can see that
\begin{align}\label{BNestimates}
\|\psi_{\epsilon}\|_{L^{2^{*}_{s}}(\R^{N})}^{2^{*}_{s}}=S_{*}^{\frac{N}{2s}}+O(\epsilon^{N}) \, \mbox{ and }\, [\psi_{\epsilon}]^{2}_{H^{s}(\R^{N})}\leq S_{*}^{\frac{N}{2s}}+O(\epsilon^{N-2s}),
\end{align}
so, in particular, we deduce
\begin{equation}\label{ve}
[v_{\epsilon}]^{2}_{H^{s}(\R^{N})}\leq S_{*}+O(\epsilon^{N-2s}).
\end{equation}
By using the assumption $(g4)$, we get $\displaystyle{\V(v_{\epsilon})\geq \frac{1}{2^{*}_{s}}+ \Gamma_{\epsilon}}$, where 
$$
\Gamma_{\epsilon}=\left(\frac{C}{q}\|v_{\epsilon}\|^{q}_{L^{q}(\R^{N})}-\frac{a}{2}\|v_{\epsilon}\|^{2}_{L^{2}(\R^{N})}\right).
$$
Our aim is to prove that
\begin{equation}\label{gammae} 
\lim_{\epsilon \rightarrow 0} \frac{\Gamma_{\epsilon}}{\epsilon^{N-2s}}=+\infty.
\end{equation}
Noting that 
\begin{align}
\|v_{\epsilon}\|^{q}_{L^{q}(\R^{N})}&\geq \frac{1}{\|\psi_{\epsilon}\|^{q}_{L^{2^{*}_{s}}(\R^{N})}} \int_{B_{1}} |U_{\epsilon}(x)|^{q} dx \nonumber \\
&=C_{1}(\epsilon) \epsilon^{N-\frac{(N-2s)q}{2}}\int_{0}^{\frac{1}{\epsilon}} \frac{r^{N-1}}{(1+r^{2})^{\frac{(N-2s)q}{2}}} dr
\end{align}
and
\begin{align}
\|v_{\epsilon}\|^{2}_{L^{2}(\R^{N})}&\leq \frac{1}{\|\psi_{\epsilon}\|^{2}_{L^{2^{*}_{s}}(\R^{N})}} \int_{B_{2}} |U_{\epsilon}(x)|^{2} dx \nonumber \\
&=C_{2}(\epsilon) \epsilon^{2s} \int_{0}^{\frac{2}{\epsilon}} \frac{r^{N-1}}{(1+r^{2})^{N-2s}} dr
\end{align}
where $\displaystyle{C_{1}(\epsilon)=\frac{\omega_{N-1} \kappa^{q}}{\|\psi_{\epsilon}\|^{q}_{L^{2^{*}_{s}}(\R^{N})}}} $ and $\displaystyle{C_{2}(\epsilon)=\frac{\omega_{N-1} \kappa^{2}}{\|\psi_{\epsilon}\|^{2}_{L^{2^{*}_{s}}(\R^{N})}}}$, we can infer that
\begin{align}\label{Ae}
\frac{\Gamma_{\epsilon}}{\epsilon^{N-2s}}&=\frac{1}{\epsilon^{N-2s}} \left[\frac{C}{q} \|v_{\epsilon}\|^{q}_{L^{q}(\R^{N})}-\frac{a}{2}\|v_{\epsilon}\|^{2}_{L^{2}(\R^{N})}\right] \nonumber \\
&\geq \frac{1}{\epsilon^{N-2s}} \left[ \frac{C}{q} C_{1}(\epsilon)\epsilon^{N-\frac{(N-2s)q}{2}}\int_{0}^{\frac{1}{\epsilon}} \frac{r^{N-1}}{(1+r^{2})^{\frac{(N-2s)q}{2}}} dr-\frac{a}{2} C_{2}(\epsilon) \epsilon^{2s} \int_{0}^{\frac{2}{\epsilon}} \frac{r^{N-1}}{(1+r^{2})^{N-2s}} dr \right] \nonumber \\
&=: \epsilon^{2s-\frac{(N-2s)q}{2}}  A(\epsilon)
\end{align}
with 
$$
A(\epsilon)=\frac{C}{q} C_{1}(\epsilon) \int_{0}^{\frac{1}{\epsilon}} \frac{r^{N-1}}{(1+r^{2})^{\frac{(N-2s)q}{2}}} dr-\frac{a}{2} C_{2}(\epsilon) \epsilon^{2s-N+\frac{(N-2s)q}{2}} \int_{0}^{\frac{2}{\epsilon}} \frac{r^{N-1}}{(1+r^{2})^{N-2s}} dr.  
$$
Let us observe that being $q>2$ if $N\geq 4s$, and $q>\frac{4s}{N-2s}$ if $N<4s$, follows that $2s-\frac{(N-2s)q}{2}<0$ for any $N>2s$. 
Therefore, in order to verify (\ref{gammae}), we are going to prove that
$$
\lim_{\epsilon \rightarrow 0} A(\epsilon)>0.
$$
By using $|\psi_{\epsilon}|\leq |U_{\epsilon}|$ and (\ref{BNestimates}), we can observe that there are $L_{1}, L_{2}>0$ such that
$$
\lim_{\epsilon \rightarrow 0} C_{1}(\epsilon)=L_{1} \mbox{ and } \lim_{\epsilon \rightarrow 0} C_{2}(\epsilon)=L_{2}.
$$
Since $q>\frac{4s}{N-2s}>\frac{N}{N-2s}$ if $N<4s$ and $q>2>\frac{N}{N-2s}$ if $N \geq 4s$, it is clear that
$$
\delta:=\int_{0}^{\infty} \frac{r^{N-1}}{(1+r^{2})^{\frac{(N-2s)q}{2}}} dr<\infty.
$$
Now we show that
$$
\lim_{\epsilon \rightarrow 0} \epsilon^{2s-N+\frac{(N-2s)q}{2}} \int_{0}^{\frac{2}{\epsilon}} \frac{r^{N-1}}{(1+r^{2})^{N-2s}} dr=0.
$$
Set
$$
C_{3}(\epsilon)=\epsilon^{2s-N+\frac{(N-2s)q}{2}} \int_{0}^{\frac{2}{\epsilon}} \frac{r^{N-1}}{(1+r^{2})^{N-2s}} dr.
$$
Fix $\epsilon \in (0, 2)$. 
Then, for $N>4s$ we have
$$
C_{3}(\epsilon)\leq \epsilon^{2s-N+\frac{(N-2s)q}{2}} \int_{0}^{1} \frac{r^{N-1}}{(1+r^{2})^{N-2s}} dr+\epsilon^{2s-N+\frac{(N-2s)q}{2}} \int_{1}^{\infty} \frac{1}{r^{N-4s+1}} dr \rightarrow 0,
$$
since $2s-N+\frac{(N-2s)q}{2}>0$ because of $q>2$.  \\
If $N=4s$ then
$$
C_{3}(\epsilon)\leq \epsilon^{s(q-2)} \int_{0}^{1} \frac{r^{4s-1}}{(1+r^{2})^{2s}} dr+\epsilon^{s(q-2)} \int_{1}^{\frac{2}{\epsilon}} \frac{1}{r} dr \rightarrow 0
$$
since $q>2$ and $\lim_{\epsilon \rightarrow 0^{+}} \epsilon^{s(q-2)}\log \epsilon=0$. \\
Finally, for $N<4s$
\begin{align*}
C_{3}(\epsilon)&\leq \epsilon^{2s-N+\frac{(N-2s)q}{2}} \int_{0}^{1} \frac{r^{N-1}}{(1+r^{2})^{N-2s}} dr \\
&+\epsilon^{2s-N+\frac{(N-2s)q}{2}} \int_{1}^{\frac{2}{\epsilon}} \frac{1}{r^{N-4s+1}} dr \\
&\leq \epsilon^{2s-N+\frac{(N-2s)q}{2}} \int_{0}^{1} \frac{r^{N-1}}{(1+r^{2})^{N-2s}} dr +\frac{2^{-N+4s}}{-N+4s} \epsilon^{-2s+\frac{(N-2s)q}{2}}  \rightarrow 0
\end{align*}
being $q>\frac{4s}{N-2s}$.

\noindent
Then 
$$
\lim_{\epsilon \rightarrow 0} A(\epsilon)=\frac{C}{q}L_{1}\delta>0,
$$
and by using (\ref{Ae}) we deduce that (\ref{gammae}) is satisfied.
As a consequence there exists $\epsilon_{0}>0$ such that 
$$
\V(v_{\epsilon})\geq \frac{1}{2^{*}_{s}} \mbox{ for all } 0<\epsilon<\epsilon_{0}.
$$
Set $\displaystyle{\omega(x)=v_{\epsilon}\left(\frac{x}{\sigma}\right)}$, where $\sigma=(\V(v_{\epsilon}))^{-\frac{1}{N}}$. Then $\V(\omega)=1$ so we get $\{u\in H^{s}(\R^{N}): \V(u)=1\}$ is not empty.\\
Now we show that $0<M<\frac{1}{2}(2^{*}_{s})^{\frac{N-2s}{N}}S_{*}$. Clearly $0\leq M<\infty$. \\
Since $\T(v_{\epsilon})=\sigma^{N-2s}\T(\omega)$ and $\V(\omega)=1$ we deduce that 
$$
M\leq \T(\omega)=\frac{\T(v_{\epsilon})}{(\V(v_{\epsilon}))^{\frac{2}{2^{*}_{s}}}}.
$$
Then, by (\ref{ve}), we can see that for any $0<\epsilon<\epsilon_{0}$ 
\begin{align}\label{mgamma}
0\leq M\leq &\frac{\frac{1}{2}[v_{\epsilon}]^{2}_{H^{s}(\R^{N})}}{\Bigl(\frac{1}{2^{*}_{s}}+\frac{C}{q}\|v_{\epsilon}\|^{q}_{L^{q}(\R^{N})}-\frac{a}{2}\|v_{\epsilon}\|^{2}_{L^{2}(\R^{N})} \Bigr)^{\frac{2}{2^{*}_{s}}}} \nonumber \\
&\leq \frac{1}{2} (2^{*}_{s})^{\frac{2}{2^{*}_{s}}} S_{*} \frac{1+O(\epsilon^{N-2s})}{(1+2^{*}_{s} \Gamma_{\epsilon})^{\frac{2}{2^{*}_{s}}}}.
\end{align}
Noting that for $p\geq 1$ it holds
$$
(1+y)^{p}\leq 1+p(1+y)^{p+1}y \mbox{ for all } y\geq -1,
$$
by (\ref{gammae}) we deduce that for all $\epsilon$ sufficiently small
$$
(1+O(\epsilon^{N-2s}))^{\frac{2^{*}_{s}}{2}}-1\leq \frac{2^{*}_{s}}{2} (1+O(\epsilon^{N-2s}))^{1+\frac{2^{*}_{s}}{2}}O(\epsilon^{N-2s})<2^{*}_{s} \Gamma_{\epsilon}
$$
that is 
\begin{equation}\label{ogamma} 
(1+O(\epsilon^{N-2s}))^{\frac{2^{*}_{s}}{2}}<1+2^{*}_{s} \Gamma_{\epsilon}.
\end{equation}
Hence, taking into account (\ref{mgamma}) and (\ref{ogamma}) we obtain 
$$
M<\frac{1}{2}(2^{*}_{s})^{\frac{N-2s}{N}}S_{*}.
$$
Finally we prove that $M>0$. We argue by contradiction, and suppose that $M=0$. Then we can find a sequence $(u_{n})\subset H^{s}(\R^{N})$ such that $\V(u_{n})=1$ and $\T(u_{n})\rightarrow 0$ as $n \rightarrow \infty$. By using (\ref{FSI}) we can see that  $\lim_{n \rightarrow \infty} \|u_{n}\|_{L^{2^{*}_{s}}(\R^{N})}=0$. 
Now, by using $(g2)$ and $(g3)$, we know that there is a constant $K>0$ such that $G(t)\leq K |t|^{2^{*}_{s}}$ for all $t \in \R$, so we deduce that $1=\V(u_{n})\leq K \|u_{n}\|^{2^{*}_{s}}_{L^{2^{*}_{s}}(\R^{N})}\rightarrow 0$ as $n \rightarrow \infty$, which gives a contradiction.
\end{proof}

\noindent
Taking into account the previous result, we obtain
\begin{lem}\label{lem2}
Under the assumptions of Theorem \ref{thm1}, there exists a minimizer $u\in H^{s}(\R^{N})$ to the problem
\begin{equation}\label{minpb}
\inf \left\{\T(u): \V(u)=1, u\in H^{s}(\R^{N})\right\}.
\end{equation}
Moreover $u$ is positive and radially symmetric. 
\end{lem}
\begin{proof}
By Lemma \ref{lem1} we know that $\{\T(u): \V(u)=1, u\in H^{s}(\R^{N})\}$ is not empty.
Now, we show the existence of a radial minimizing sequence.
Let $(u_{n})\subset H^{s}(\R^{N})$ be a sequence such that $\V(u_{n})=1$ for all $n \in \N$, and $\T(u_{n}) \rightarrow M$ as $n \rightarrow \infty$. \\
Since we know (see \cite{Park}) that
$$
[u^{*}]_{H^{s}(\R^{N})}\leq [u]_{H^{s}(\R^{N})} \mbox{ for all } u\in H^{s}(\R^{N}), 
$$
where $u^{*}$ is the symmetric rearrangement of $|u|$, we can suppose that $u_{n}$ belongs to $H^{s}_{r}(\R^{N})$ and that $u_{n}$ is positive.\\
By using (\ref{FSI}) we can see that $(u_{n})$ is bounded in $L^{2^{*}_{s}}(\R^{N})$.\\
From $(g2)$ and $(g3)$, we know that there exists $K>0$ such that 
\begin{equation}\label{Gbound}
G(t)\leq K|t|^{2^{*}}-\frac{a}{4}t^{2} \mbox{ for all } t\in \R.
\end{equation}
Then, putting together $\V(u_{n})=1$, (\ref{Gbound}) and $(u_{n})$ is bounded in $L^{2^{*}_{s}}(\R^{N})$, we get $(u_{n})$ is bounded in $L^{2}(\R^{N})$. As a consequence, $(u_{n})$ is bounded in $H^{s}(\R^{N})$, so $u_{n} \rightharpoonup u$ in $H^{s}(\R^{N})$, and by Theorem \ref{Lions}
we have $u_{n} \rightarrow u$ in $L^{q}(\R^{N})$ and a.e. in $\R^{N}$.\\
Let $v_{n}=u_{n}-u$. Then, the weak convergence in $H^{s}(\R^{N})$ (which is a Hilbert space), implies that as $n \rightarrow \infty$
\begin{equation}\label{Tun}
\T(u_{n})=\T(v_{n})+\T(u)+o(1),
\end{equation}
while the Brezis-Lieb Lemma \cite{BL} yields
\begin{equation}\label{un1}
\|u_{n}\|_{L^{2^{*}_{s}}(\R^{N})}^{2^{*}_{s}}=\|v_{n}\|_{L^{2^{*}_{s}}(\R^{N})}^{2^{*}_{s}}+\|u\|_{L^{2^{*}_{s}}(\R^{N})}^{2^{*}_{s}}+o(1)
\end{equation}
and
\begin{equation}\label{un2}
\|u_{n}\|_{L^{2}(\R^{N})}^{2}=\|v_{n}\|_{L^{2}(\R^{N})}^{2}+\|u\|_{L^{2}(\R^{N})}^{2}+o(1),
\end{equation}
as $n \rightarrow \infty$.
Let $f(t)=g(t)-t^{{2^{*}_{s}}-1}+at$ and $\displaystyle{F(t)=\int_{0}^{t} f(\xi) d\xi}$. 
By using $(g2)$ and $(g3)$ we know that
\begin{align*}
\lim_{|t| \rightarrow 0} \frac{F(t)}{|t|^{2}}=0 \, \mbox{ and }\, \lim_{|t|\rightarrow \infty} \frac{F(t)}{|t|^{2^{*}_{s}}}=0,
\end{align*}
and by the boundedness of $(u_{n})$ in $H^{s}(\R^{N})$ follows that we can apply Lemma \ref{CW} to deduce that as $n \rightarrow \infty$
\begin{equation}\label{intF}
\int_{\R^{N}} F(v_{n}) dx=o(1) \mbox{ and } \int_{\R^{N}} F(u_{n}) dx=\int_{\R^{N}} F(u) dx+o(1). 
\end{equation}
Putting together (\ref{un1}), (\ref{un2}) and (\ref{intF}) we get
\begin{equation}\label{Vun}
\V(u_{n}) = \V(v_{n}) + \V(u) + o(1) \mbox{ as } n\rightarrow \infty. 
\end{equation}
Set $\tau_{n}=\T(v_{n})$, $\tau=\T(u)$, $\nu_{n}=\V(v_{n})$  and $\nu=\V(u)$.
Then, with this notation, (\ref{Tun}) and (\ref{Vun}) read
\begin{equation}\label{taunu}
\tau_{n}=M-\tau+o(1) \, \mbox{ and } \, \nu_{n}=1-\nu+o(1).
\end{equation}
In order to prove the existence of a minimizer of (\ref{minpb}), it is enough to prove that $\nu=1$. \\
Firstly, let us observe that, if $u_{\sigma}(x)=u(\frac{x}{\sigma})$, then $\T(u_{\sigma})=\sigma^{N-2s} \T(u)$ and $\V(u_{\sigma})=\sigma^{N} \V(u)$, so we have
\begin{equation}\label{TMV}
\T(u)\geq M(\V(u))^{\frac{N-2s}{N}},
\end{equation}
for all $u\in H^{s}(\R^{N})$ and $\V(u)\geq 0$.\\
If by contradiction $\nu>1$, then $\tau\geq M\nu^{\frac{N-2s}{N}}>M$, which is in contrast with the fact that $\tau\leq M$. Then $\nu\leq 1$. \\
If $\nu<0$, then $\nu_{n}>1-\frac{\nu}{2}>1$ for $n$ sufficiently large. By using (\ref{TMV}), we can see that
$$
\tau_{n}\geq M\nu_{n}^{\frac{N-2s}{N}}>M\Bigl(1-\frac{\nu}{2} \Bigr)^{\frac{N-2s}{N}}
$$
for $n$ sufficiently large, and this gives a contradiction since $\tau_{n}\leq M+o(1)$ by (\ref{taunu}). Then we have $\nu\in [0, 1]$.\\
If by contradiction $\nu\in [0, 1)$, then $\nu_{n}>0$ for  all $n$ big enough. By (\ref{TMV}) we have 
$$
\tau_{n}\geq M\nu_{n}^{\frac{N-2s}{N}} \mbox{ and } \tau\geq M\nu^{\frac{N-2s}{N}}.
$$
This yields
\begin{align*}
M&=\lim_{n \rightarrow \infty} \tau+\tau_{n}\\
&\geq \lim_{n \rightarrow \infty} M \left [\nu^{\frac{N-2s}{N}}+\nu_{n}^{\frac{N-2s}{N}}\right] \\
&=M \left[\nu^{\frac{N-2s}{N}}+(1-\nu)^{\frac{N-2s}{N}}\right] \\
&\geq M[\nu+1-\nu]=M.
\end{align*}
As a consequence, $\nu^{\frac{N-2s}{N}}+(1-\nu)^{\frac{N-2s}{N}}=1$, and by using the fact that $\nu\in [0, 1)$, we get $\nu=0$. Then $u=0$ and $\tau_{n} \rightarrow M$ as $n \rightarrow \infty$.\\
Moreover, by (\ref{taunu}), $\nu_{n} \rightarrow 1$ as $n \rightarrow \infty$, and by  the definition of $F$ and (\ref{intF}) we can infer that 
$$
\|v_{n}\|_{L^{2^{*}_{s}}(\R^{N})}^{2^{*}_{s}}=2^{*}_{s}+\frac{2^{*}_{s}}{2}  a\|v_{n}\|_{L^{2}(\R^{N})}^{2}+o(1),
$$
which implies that
$$
\limsup_{n \rightarrow \infty} \|v_{n}\|^{2}_{L^{2^{*}_{s}}(\R^{N})}\geq (2^{*}_{s})^{\frac{2}{2^{*}_{s}}}.
$$
Therefore, recalling that $\displaystyle{S_{*}=\inf_{u\in H^{s}(\R^{N})\setminus \{0\}} \frac{[u]^{2}_{H^{s}(\R^{N})}}{\|u\|^{2}_{L^{2^{*}_{s}}(\R^{N})}}}$, we get
\begin{align*}
M&=\frac{1}{2} \lim_{n \rightarrow \infty} [v_{n}]^{2}_{H^{s}(\R^{N})}\\
&\geq \frac{1}{2}(2^{*}_{s})^{\frac{2}{2^{*}_{s}}} \liminf_{n \rightarrow \infty} \frac{[v_{n}]^{2}_{H^{s}(\R^{N})}}{\|v_{n}\|^{2}_{L^{2^{*}_{s}}(\R^{N})}}\\
&\geq \frac{1}{2}(2^{*}_{s})^{\frac{2}{2^{*}_{s}}} S_{*},
\end{align*}
which gives a contradiction by Lemma \ref{lem1}. \\
Then, we have proved that $\nu=1$, that is $u\in H^{s}_{r}(\R^{N})$ is a positive minimizer of (\ref{minpb}).
\end{proof}

\noindent
Before giving the proof of Theorem \ref{thm1}, we recall the following Pohozaev identity:
\begin{thm}
Let $g: \R\rightarrow \R$ be a function satisfying $(g1), (g2)$ and $(g3)$, and $u\in H^{s}(\R^{N})$ be a weak solution to $(-\Delta)^{s}u=g(u)$ in $\R^{N}$.
Then $u$ satisfies the Pohozaev identity
\begin{equation}\label{Pohid}
\frac{N-2s}{2}[u]_{H^{s}(\R^{N})}^{2}=N\int_{\R^{N}} G(u) dx.
\end{equation}
\end{thm}
\begin{proof}
Let us observe that the proof of (\ref{Pohid}) has been established in \cite{CW} when the nonlinearity $g$ is subcritical.
However, we can see that the arguments in \cite{CW} work again in our case.  
For the reader's convenience we give a sketch of the proof. \\
Clearly, $u$ satisfies $(-\Delta)^{s}u=q(x)u$ in $\R^{N}$, where $q(x)=g(u(x))/u(x)\in L^{\frac{N}{2s}}_{loc}(\R^{N})$ in view of $(g2)$-$(g3)$.
Then, by using similar arguments as in \cite{BCPS}, we deduce that $u\in L^{\infty}_{loc}(\R^{N})$. Proceeding as in \cite{Cabsire, FQT} and by using $(g1)$, we can infer that $u\in C^{2}(\R^{N})$.
Now, transforming (\ref{P}) into a local problem via extension method \cite{CS}, we can see that the $s$-harmonic extension $v=E_{s}(u)\in C^{2}(\R^{N+1}_{+})$ of $u$ solves
\begin{equation}\label{Pext}
\left\{
\begin{array}{ll}
\dive(y^{1-2s} \nabla v)=0 & \mbox{ in } \R^{N+1}_{+} \\
-\lim_{y\rightarrow 0^{+}} y^{1-2s}\frac{\partial v}{\partial y}=\kappa_{s} g(v)  & \mbox{ on } \R^{N} \\
v(x, 0)=u(x) & \mbox{ on } \R^{N}
\end{array}
\right.
\end{equation}
and satisfies
\begin{equation}\label{extension}
\int_{\R^{N+1}_{+}} y^{1-2s}|\nabla v|^{2}dx dy=[u]_{H^{s}(\R^{N})}^{2}<\infty,
\end{equation}
where $\kappa_{s}$ is a positive constant (see \cite{CS}). Without loss of generality we may assume that $\kappa_{s}=1$.\\
For any $R>0$ and $\delta\in (0, R)$, we define 
$$
D_{R, \delta}^{+}=\{(x, y)\in \R^{N}\times [\delta, \infty): |(x, y)|\leq R\},
$$
and we denote by $\partial D^{1}_{R, \delta}=\{(x, y)\in \R^{N}\times \{y=\delta\}: |x|^{2}\leq R^{2}-\delta^{2}\}$ and 
 $\partial D^{2}_{R, \delta}=\{(x, y)\in \R^{N}\times [\delta, \infty): |(x, y)|= R \}$.\\
Multiplying the equation (\ref{Pext}) by $((x, y), \nabla v)$ and integrating on $D^{+}_{R, \delta}$ we get
\begin{align}\label{locPI}
0&=\int_{D_{R, \delta}^{+}} \dive(y^{1-2s} \nabla v) ((x, y), \nabla v) dx dy \nonumber\\
&=\int_{\partial D^{1}_{R, \delta}} y^{1-2s} \left[(x, \nabla_{x}v)(-v_{y})+\frac{y}{2}|\nabla v|^{2}\right] d\sigma \nonumber\\&+\int_{\partial D^{2}_{R, \delta}} y^{1-2s}\left[\frac{1}{R}|((x, y), \nabla v)|^{2}-\frac{R}{2}|\nabla v|^{2}\right] d\sigma +\frac{N-2s}{2}\int_{D_{R, \delta}^{+}} y^{1-2s}|\nabla v|^{2} dxdy \nonumber \\
&=:A_{1}(R, \delta)+A_{2}(R, \delta)+A_{3}(R, \delta).
\end{align}
By using the boundary condition in (\ref{Pext}) and the divergence theorem we can see that
\begin{equation}\label{gr1}
\lim_{\delta\rightarrow 0}\int_{\partial D^{1}_{R, \delta}} y^{1-2s} (x, \nabla_{x}v)(-v_{y}) d\sigma=R\int_{\partial B_{R}} G(v)dS-N\int_{B_{R}} G(v) dx
\end{equation}
where $B_{R}=\{(x, 0): |x|\leq R\}$.\\
Clearly
\begin{equation}\label{gr2}
\lim_{\delta\rightarrow 0} \int_{\partial D^{1}_{R, \delta}} y^{1-2s} y|\nabla v|^{2} d\sigma=0.
\end{equation}
Taking into account $G(v)\in L^{1}(\R^{N})$ and (\ref{extension}),
we can find a sequence $R_{n}\rightarrow \infty$ as $n\rightarrow \infty$ such that
\begin{equation}\label{gr3}
\lim_{n\rightarrow \infty} R_{n}\int_{\partial B_{R_{n}}^{+}} G(v)dS= 0,
\end{equation}
\begin{equation}\label{gr4}
\lim_{n\rightarrow \infty} A_{2}(R_{n}, \delta)=0 \mbox{ for all } \delta>0.
\end{equation}
Putting together (\ref{gr1})-(\ref{gr4}), we deduce that
\begin{align}\begin{split}\label{gr5}
&\lim_{n\rightarrow \infty}\lim_{\delta\rightarrow 0}A_{1}(R_{n}, \delta)=-N\int_{\R^{N}} G(v) dx \\
&\lim_{n\rightarrow \infty}\lim_{\delta\rightarrow 0}A_{2}(R_{n}, \delta)=0.
\end{split}\end{align} 
Then, by using (\ref{gr5}) and the fact that 
$$
\lim_{n\rightarrow \infty}\lim_{\delta\rightarrow 0}A_{3}(R_{n}, \delta)=\frac{N-2s}{2}\int_{\R^{N+1}_{+}} y^{1-2s}|\nabla v|^{2} dx dy,
$$
from (\ref{locPI}) and (\ref{extension}) we get (\ref{Pohid}).
\end{proof}

\noindent
Now we are able to prove the main result of this section:
\begin{proof}[Proof of Theorem \ref{thm1}]
Let us show that there exists a least energy solution $\omega(x)$ of (\ref{P}) 
and that $\displaystyle{m=\frac{s}{N}\Bigl(\frac{N-2s}{2N}\Bigr)^{\frac{N-2s}{2}}(2M)^{\frac{N}{2s}}}$, where $m$ is the least energy of (\ref{P}); see (\ref{LEL}).\\
Let us define the following sets
$$
\mathcal{S}=\{ v\in H^{s}(\R^{N}): \V(v)=1\}
$$
and 
$$
\mathcal{P}=\{v\in H^{s}(\R^{N})\setminus \{0\}:  \J(v)=0\}.
$$
where 
$$
\J(v)=\frac{1}{2}[v]^{2}_{H^{s}(\R^{N})}-\frac{N}{N-2s}\int_{\R^{N}} G(v) dx\in C^{1}(H^{s}(\R^{N}), \R). 
$$
Then, there exists a one-to-one correspondence $\Phi: \mathcal{S} \rightarrow \mathcal{P}$ between $\mathcal{S}$ and $\mathcal{P}$ defined by
$$
(\Phi(v))(x)=v\Bigl(\frac{x}{\tau_{u}}\Bigr)\, \mbox{ where } \tau_{u}=\Bigl(\frac{N-2s}{2N}\Bigr)^{\frac{1}{2s}}[v]^{\frac{1}{s}}_{H^{s}(\R^{N})}. 
$$
Let us notice that for any $v \in \mathcal{S}$
$$
\I(\Phi(v))=\tau_{u}^{N-2s}\T(v)-\tau^{N}_{u}\V(v)=\frac{s}{N}\Bigl(\frac{N-2s}{2N}\Bigr)^{\frac{N-2s}{2s}} [v]^{\frac{N}{s}}_{H^{s}(\R^{N})}.
$$
Then
$$
\inf_{v\in \mathcal{P}} \I(v)=\inf_{v \in \mathcal{P}} \I(\Phi(v))=\frac{s}{N}\Bigl(\frac{N-2s}{2N}\Bigr)^{\frac{N-2s}{2s}} \inf_{v\in \mathcal{S}} [v]^{\frac{N}{s}}_{H^{s}(\R^{N})}.
$$
By using Lemma \ref{lem2}, we know that $u \in \mathcal{S}$ and that 
$$
\inf_{v \in \mathcal{S}} [v]^{2}_{H^{s}(\R^{N})}=[u]^{2}_{H^{s}(\R^{N})}=2M.
$$
As a consequence  
$$
\inf_{v\in \mathcal{P}} \I(v)=\I(\Phi(u))=\frac{s}{N}\Bigl(\frac{N-2s}{2N}\Bigr)^{\frac{N-2s}{2s}}(2M)^{\frac{N}{2s}}.
$$ 
Let $\omega=\Phi(u)$. By Lagrange multipliers Theorem, there exists $\lambda\in \R$ such that 
$$
\langle \I'(\omega), \varphi \rangle=\lambda \langle \J'(\omega), \varphi \rangle \mbox{ for any } \varphi \in H^{s}(\R^{N}).
$$ 
This means that $\omega$ is a weak solution to
$$
(1-\lambda)(-\Delta)^{s}\omega=\Bigl(1-\lambda \frac{N}{N-2s}\Bigr)g(\omega) \mbox{ in } \R^{N},
$$
and in view of (\ref{Pohid}), $\omega$ satisfies the following Pohozaev identity
\begin{align}\label{P1}
(1-\lambda)\frac{N-2s}{2}[\omega]^{2}_{H^{s}(\R^{N})}=N\Bigl(1-\lambda \frac{N}{N-2s}\Bigr) \int_{\R^{N}} G(\omega) dx. 
\end{align}
Since $\omega=\Phi(u) \in \mathcal{P}$, we also know that
\begin{align}\label{P2}
\frac{N-2s}{2}[\omega]^{2}_{H^{s}(\R^{N})}=N \int_{\R^{N}} G(\omega) dx.
\end{align}
Putting together (\ref{P1}) and (\ref{P2}) we can see that 
$$
\lambda\Bigl(1-\frac{N}{N-2s}\Bigr)\int_{\R^{N}} G(\omega) dx=0,
$$
so we have $\lambda=0$ and $\I'(\omega)=0$. Then $\omega$ is a least energy solution to (\ref{P}), and $\displaystyle{m=\frac{s}{N}\Bigl(\frac{N-2s}{2N}\Bigr)^{\frac{N-2s}{2}}(2M)^{\frac{N}{2s}}}$.

\end{proof}

\section{Mountain pass characterization of least energy solutions}

\noindent
In this section we give the proof of Theorem \ref{thm2}.
Firstly we prove that
\begin{lem}\label{lem3}
$\I$ has a mountain pass geometry, that is:
\begin{compactenum}[(i)]
\item $\I(0)=0$;
\item There exist $\rho>0$, $\eta>0$ such that $\I(u)\geq \eta$ for all $\|u\|_{H^{s}(\R^{N})}=\rho$;
\item There exist $u_{0}\in H^{s}(\R^{N})$ such that $\|u_{0}\|_{H^{s}(\R^{N})}>\rho$ and $\I(u_{0
})<0$.
\end{compactenum}
Then
$$
c=\inf_{\gamma\in \Gamma} \max_{t\in [0, 1]} \I(\gamma(t))
$$
where
$$
\Gamma=\{\gamma\in C([0, 1], H^{s}(\R^{N})): \gamma(0)=0 \mbox{ and } \I(\gamma(1))<0\},
$$
is well defined.
\end{lem}
\begin{proof}
Clearly $\I(0)=0$.
By using the assumptions $(g2)$ and $(g3)$, there exists a positive constant $C_{a}$ such that 
\begin{equation}\label{Gt}
G(t)\leq -\frac{a}{4}t^{2}+C_{a}|t|^{2^{*}_{s}} \mbox{ for all } t \in \R.
\end{equation}
Then, by using (\ref{FSI}) and (\ref{Gt}), we get
\begin{align*}
\I(u)&\geq \frac{1}{2}[u]^{2}_{H^{s}(\R^{N})}+\frac{a}{4}\|u\|^{2}_{L^{2}(\R^{N})}-C_{a}\|u\|_{L^{2^{*}_{s}}(\R^{N})}^{2^{*}_{s}} \\
&\geq \min\Bigl\{\frac{1}{2}, \frac{a}{4}\Bigr\}\|u\|^{2}_{H^{s}(\R^{N})}-C_{a} \|u\|_{L^{2^{*}_{s}}(\R^{N})}^{2^{*}_{s}} \\
&\geq \min\Bigl\{\frac{1}{2}, \frac{a}{4}\Bigr\} \|u\|^{2}_{H^{s}(\R^{N})}-C_{a}S_{*}^{\frac{2^{*}_{s}}{2}} \|u\|_{H^{s}(\R^{N})}^{2^{*}_{s}},
\end{align*}
which implies that there exist $\rho>0$ and $\eta>0$ such that $\I(u)\geq \eta$ for all $\|u\|_{H^{s}(\R^{N})}= \rho$. \\
By using $(g2)$ and $(g3)$ again, there exists $C'_{a}>0$ such that 
$$
G(t)\geq -\frac{C'_{a}}{2}t^{2}+\frac{1}{2 2^{*}_{s}}|t|^{2^{*}_{s}} \mbox{ for all } t \in \R.
$$
Then, for any $u\in H^{s}(\R^{N})\setminus \{0\}$ and $t>0$
$$
\I(tu)\leq \frac{t^{2}}{2}[u]^{2}_{H^{s}(\R^{N})}+\frac{C'_{a}t^{2}}{2}\|u\|^{2}_{L^{2}(\R^{N})}-\frac{t^{2^{*}_{s}}}{2 2^{*}_{s}}\|u\|^{2^{*}_{s}}_{L^{2^{*}_{s}}(\R^{N})} \rightarrow -\infty \mbox{ as } t \rightarrow +\infty,
$$
so we can find $u_{0}\in H^{s}(\R^{N})$ such that $\|u_{0}\|_{H^{s}(\R^{N})}>\rho$ and $\I(u_{0})<0$.

\end{proof}

\begin{proof}[Proof of Theorem \ref{thm2}]
We begin proving that there exists $\gamma \in C([0,1], H^{s}(\R^{N}))$ such that $\gamma(0)=0$, $\I(\gamma(1))<0$, $\omega\in \gamma([0,1])$ and
$$
\max_{t\in [0,1]} \I(\gamma(t)) = m. 
$$
Let
\begin{equation*}
\gamma(t)(x)= 
\left\{
\begin{array}{ll}
\omega(\frac{x}{t}) &\mbox{ for } t>0\\
0 &\mbox{ for } t=0.   
\end{array}
\right.
\end{equation*}
Then we can see  
\begin{align}\label{Igamma}
&\|\gamma(t)\|^{2}_{H^{s}(\R^{N})}=t^{N-2s}[\omega]^{2}_{H^{s}(\R^{N})}+t^{N}\|\omega\|^{2}_{L^{2}(\R^{N})} \nonumber \\
&\I(\gamma(t))= \frac{t^{N-2s}}{2} [\omega]^{2}_{H^{s}(\R^{N})}- t^{N} \int_{\R^{N}} G(\omega) \,dx. 
\end{align}
Hence $\gamma\in C([0, \infty), H^{s}(\R^{N}))$. \\
By using the Pohozaev identity (\ref{Pohid}), we know that 
\begin{equation}\label{Gomega}
\int_{\R^{N}}G(\omega) dx=\frac{N-2s}{2N} [\omega]^{2}_{H^{s}(\R^{N})}>0. 
\end{equation}
Putting together (\ref{Igamma}) and (\ref{Gomega}) we have
$$
\frac{d}{dt} \I(\gamma(t))=t^{N-2s-1}(1-t^{2s})\frac{N-2s}{2}[\omega]^{2}_{H^{s}(\R^{N})}, 
$$
and in particular 
$$
\frac{d}{dt} \I(\gamma(t))>0 \mbox{ for } t\in (0, 1) \,
\mbox{ and } \, \frac{d}{dt} \I(\gamma(t))<0 \mbox{ for } t>1.
$$
Thus, for  $L>1$ sufficiently large, there exists a path $\gamma(t): [0, L] \rightarrow H^{s}(\R^{N})$ such that $\gamma(0)=0$, $\I(\gamma(L))<0$, $\omega\in \gamma([0,L])$ and
$$
\max_{t\in [0,L]} \I(\gamma(t)) = m. 
$$
After a suitable scale change in $t$, we can get the desired path $\gamma\in \Gamma$.
Hence, $c\leq m$. From the proof of Theorem \ref{thm1}, we can deduce that $m=\inf_{v \in \mathcal{P}} \I(v)$.\\
Now, let
$$
\mathcal{H}(u)=\frac{N-2s}{2}[u]^{2}_{H^{s}(\R^{N})}-N\int_{\R^{N}} G(u) dx=N\I(u)-s[u]^{2}_{H^{s}(\R^{N})}.
$$
Modifying slightly the arguments of the proof of Lemma \ref{lem3}, it is possible to show that  
there exists $\rho_{0}>0$ such that if $0<\|u\|_{H^{s}(\R^{N})}\leq \rho_{0}$ then $\mathcal{H}(u)>0$.
Then, for any $\gamma \in \Gamma$, we have $\gamma(0)=0$ and $\mathcal{H}(\gamma(1))\leq N\I(\gamma(1))<0$. As a consequence there exists $t_{0}\in [0, 1]$ such that 
$$
\|\gamma(t_{0})\|_{H^{s}(\R^{N})}>\rho_{0} \mbox{ and }
\mathcal{H}(\gamma(t_{0}))=0.
$$
Since $\gamma(t_{0})\in \gamma([0, 1])\cap \mathcal{P}$, we have $\gamma([0, 1])\cap \mathcal{P}\neq \emptyset$. Hence $c=m$.
\end{proof}

\smallskip
\noindent {\bf Acknowledgements.}
The author is grateful to the referee for his/her comments for improvement of the article.

\end{document}